\documentclass[11pt]{amsart}
\usepackage{amsmath,amsthm,amsfonts,amssymb, array, epsfig, float, caption} %titlesec,
\usepackage[margin=1in]{geometry} %[left=0.7in,top=0.5in,right=0.7in]
\usepackage{graphicx}
\usepackage{subcaption}
\usepackage{hyperref}

\numberwithin{equation}{section}

\newtheorem{thm}{Theorem}[section]
\newtheorem{lem}[thm]{Lemma}
\newtheorem{prop}[thm]{Proposition}
\newtheorem{cor}[thm]{Corollary}
\newtheorem{conj}[thm]{Conjecture}

\newtheorem*{thm*}{Theorem}
\newtheorem*{cor*}{Corollary}

\theoremstyle{definition}

\newtheorem{thml}{Theorem}

\theoremstyle{remark}
\newtheorem{rem}[thm]{Remark}
\newtheorem*{rem*}{Remark}

\newcommand{\eps}{\varepsilon}
\newcommand{\wt}[1]{\widetilde{#1}}
\newcommand{\R}{\mathbb{R}}
\newcommand{\N}{\mathbb{N}}
\newcommand{\cT}{\mathcal{T}}

\DeclareMathOperator{\inj}{inj}
\DeclareMathOperator{\sys}{sys}
\DeclareMathOperator{\vol}{vol}
\DeclareMathOperator{\diam}{diam}
\DeclareMathOperator{\grad}{grad}

\usepackage{color}

\begin{document}

\title[Flexibility in fixed conformal classes]{Flexibility of geometrical and dynamical data in fixed conformal classes}
\author{Thomas Barthelm\'e}
\address{Department of Mathematics and Statistics, Queen's University, Kingston, ON}
\email{thomas.barthelme@queensu.ca}

\author{Alena Erchenko}
\address{Department of Mathematics, Pennsylvania State University, State College, PA}
\email{axe930@psu.edu}

\begin{abstract}
Consider a smooth closed surface $M$ of fixed genus $\geqslant 2$ with a hyperbolic metric $\sigma$ of total area $A$. In this article, we study the behavior of geometric and dynamical characteristics (e.g., diameter, Laplace spectrum, Gaussian curvature and entropies) of nonpositively curved smooth metrics with total area $A$ conformally equivalent to $\sigma$. For such metrics, we show that the diameter is bounded above and the Laplace spectrum is bounded below away from zero by constants which depend on $\sigma$. On the other hand, we prove that the metric entropy of the geodesic flow with respect to the Liouville measure is flexible. Consequently, we also provide the first known example showing that the bottom of the $L^2$-spectrum of the Laplacian cannot be bounded from above by a function of the metric entropy.
We also provide examples showing that our conditions are essential for the established bounds.
\end{abstract}

\maketitle

% \tableofcontents

\section{Introduction}

% \subsection{Summary of the result}

Let $M$ be a closed surface equipped with a Riemannian metric $\sigma$. An old and classical problem in geometry is to wonder how dynamical or geometrical invariants can change when one deforms the metric $\sigma$ in a certain class of metric. Among the invariants of interests are the Laplace spectrum, both on the surface, and the $L^2$-spectrum on the universal cover, the systole, the entropies (metric, harmonic, or topological) of the geodesic flow, or the Kaimanovich entropy and linear drift of the Brownian motion.
Since anything can be changed by scaling, a necessary condition to make the problem non-trivial is to fix, for instance, the volume.
There has been a lot of work done around that problem, in particular when the deformation is taken among all the negatively (or nonpositively) curved metrics. Another often studied class is the conformal class of a fixed metric. There seem, however, to have been relatively little work for metrics in the intersection of these two classes. Hence, in the present article, we aim to investigate how the geometric and dynamical invariants behave when we consider all negatively, or non positively, curved metrics in a fixed conformal class, with fixed area.

% Then the first natural class of metrics to consider the deformation in is all possible metrics on $M$ with fixed volume. Not many restrictions can be expected in that class. However, one geometric invariant already admits a restriction: Yang and Yau \cite{YY80} (and Hersch \cite{He74} for the sphere) proved that $\lambda_1$, the first non-zero eigenvalue of the Laplacian on $M$, is bounded above by $4\pi(4-\chi(M))/A$, where $A$ is the area of $M$ and $\chi(M)$ the Euler characteristic.
% When one restricts further the class of metrics that are considered, the flexibility in the possible values of more and more geometrical or dynamical invariants is further constraint. Among the family of particular interest are negatively (or non positively) curved metrics on $M$, hyperbolic metrics, and metrics in a fixed conformal class.

In \cite{K}, Katok proved that for any hyperbolic metric $\sigma$ on $M$ with total area $A$ and any negatively curved metric $g=e^{2u}\sigma$, with total area $A$, we have 
\begin{equation*} %\label{relation_entropies_conform}
h_{\mu}(g)\leqslant h_{\mu}(\sigma)\int_M e^u \, \frac{d v_{\sigma}}{A}\quad \text{ and } \quad h_{top}(\sigma)\left(\int_M e^u \, \frac{d v_{\sigma}}{A}\right)^{-1}\leqslant h_{top}(g), 
\end{equation*}
where $h_{\mu}(g)$ and $h_{top}(g)$ are, respectively, the metric entropy with respect to the Liouville measure, and the topological entropy of the geodesic flow.
In particular, Katok obtains that, for any negatively curved metric $g$ with total volume $A$, we have
\begin{equation}\label{relation_entropies}
h_{\mu}(g)\leqslant \left(\frac{2\pi|\chi(M)|}{A}\right)^{\frac{1}{2}}\leqslant h_{top}(g).
\end{equation}
Furthermore, he shows that either equality above holds if and only if $g$ is a hyberbolic metric.
Moreover, in \cite{EK} the second author and Katok proved that equation (\ref{relation_entropies}) gives the only restriction on the possible pair of entropies in the class of negatively curved metrics with total area $A$. They do not, however, control the conformal classes of the metrics that they build.

Another famous set of inequalities (due to Guivarch \cite{Gu80} and Ledrappier \cite{Led90,Led10}), that are satisfied for any metrics on a compact manifold, is the following
\begin{equation*} %\label{eq_list_inequalities}
 4\wt\lambda_1(g) \leq h(g)\leq l(g) h_{top}(g)\leq h_{top}(g)^2,
\end{equation*}
where $\wt\lambda_1(g)$ is the bottom of the $L^2$-spectrum of the Laplacian on the universal cover $\wt M$, and $h(g)$ and $l(g)$ are respectively the Kaimanovich entropy and the linear drift of the associated Brownian motion (see, for instance \cite{Led10} for the definitions). Moreover, in dimension $2$, when $g$ is assumed to be negatively curved, we have a strong rigidity result: any equality in the chain of inequalities above imply that $g$ is hyperbolic \cite{Led90}. When considering a fixed conformal class, Ledrappier \cite{L} proved that, for any hyperbolic metric $\sigma$ on $M$ with total area $A$ and any negatively curved metric $g=e^{2u}\sigma$, we have 

\begin{equation}\label{relation_drifts_conform}
l(g)\leqslant l(\sigma)\int_M e^u \, \frac{d v_{\sigma}}{A}, \qquad h(g)=h(\sigma) \quad \text{ and } \quad h_{harm}(g)\geqslant h_{harm}(\sigma)\left(\int_M e^u \, \frac{d v_{\sigma}}{A}\right)^{-1},
\end{equation}
where $h_{harm}(g)$ is the entropy of the harmonic measure, and is equal to $h(g)/l(g)$ \cite{L}. Moreover, $h_{top}(g) \geqslant h_{harm}(g)$ (by the variational principle) and equality holds if and only if $g$ is a hyperbolic metric \cite{L}.
Notice that, in spite of all these related results, the relationship, if any, between $\wt\lambda_1$ and the entropy of the Liouville measure, $h_{\mu}$, has remained very mysterious so far. In what appears to be one of the first steps in this direction, we will show (Theorem \ref{thmintro_example_small_metric_entropy}) that $h_{\mu}$ cannot be an upper bound for $\wt\lambda_1$.

Our aim in this study is to try to uncover the further restrictions, if any, for all of the above invariants that may appear when a conformal class is fixed, in the spirit of A.~Katok flexibility program (see \cite{BKRH,Er17,EK}). Notice that bounds on all of these invariants can be obtained if one assumes a lower bound on the curvature. However, the curvature is unbounded below in a conformal class, hence none of these results are available to us.

For any positive constant $A$, we let $[g]_A^{\leqslant}$ (resp. $[g]_A^{<}$) be the family of nonpositively (resp. negatively) curved metrics conformally equivalent to $g$ and with total area $A$. In the remainder of this text, a hyperbolic metric will always refer to a metric of constant negative curvature equal to $-1$.
Our first result, which can be seen as a slight generalization of Schwarz Lemma (see, for instance, \cite{Y}), is the following

\begin{thml}[Theorem \ref{bound_average_below_max_above}] \label{thmintro_upper_bound_on_u}
Let $(M,\sigma)$ be a hyperbolic surface of genus $\geqslant 2$ and total area $A$. Then, there exists a positive constant $C = C(\sigma)$ such that for any smooth function $u\colon M \rightarrow \R$ such that $g= e^{2u}\sigma\in [\sigma]_A^{\leqslant}$ we have $u(x)\leqslant C$ for any $x\in M$.
\end{thml}

This result admits many corollaries. First, we obtain an upper bound on the diameter of $M$ for any $g\in [\sigma]_A^{\leqslant}$ depending only on $\sigma$ (independently of that result, we also obtain an upper bound on the diameter under the weaker assumption that $g$ has no conjugate point, see Theorem \ref{thm_diam_no_conjugate}). Second, thanks to the Min-Max principle (see Subsection \ref{subsec_upper_bound_conformal}), we also get a lower bound on the Laplace spectrum:
\begin{cor*}[Corollary \ref{bound_Laplace_spectrum_conformal}]
Let $(M,\sigma)$ be a hyperbolic surface of genus $\geq2$ and total area $A$.

There exists a constant $C>0$ such that the following holds:
For any $g= e^{2u}\sigma \in [\sigma]_A^{\leqslant}$ and any $k$, we have
\[
\lambda_k(g)\geqslant C\lambda_k(\sigma). 
\]

In particular, the function $g \mapsto \lambda_1(g)$ is uniformly bounded away from $0$ on the space $[\sigma]_A^{\leqslant}$.

Similarly, on the universal cover, we have 
\[
 \wt \lambda_1(\wt g) \geqslant C \wt \lambda_1(\wt \sigma).
\]
\end{cor*}
Let us stress again that, as opposed to, for instance, Li and Yau's lower bound for the Laplace spectrum (see \cite[Theorem 4]{SY}), our result does not require a lower bound on the Gaussian curvature (see Section \ref{sec_flexibility} for an example of metric in a fixed conformal class with arbitrarily negative Gaussian curvature).

We also study in more details a certain family $g_{\eps}$ of metrics build in \cite{EK}. This family is such that the metric entropy of $g_{\eps}$ comes arbitrary close to zero (for $\eps$ going to $0$) and the topological entropy is bounded above by a fixed constant. By \cite{L}, the linear drift of these examples is also bounded away from zero by a fixed constant.

We prove that the family of hyperbolic metrics $\sigma_{\eps}$ in the conformal class of $g_{\eps}$ must stay in a compact part of the Teichm\"uller space (see Theorem \ref{example_metric_entropy} and its proof). As a result of this study and our Corollary above, we obtain the first known example of a negatively curved metric $g$ such that $h_{\mu}(g)^2 < 4\wt\lambda_1(g)$:
\begin{thml}[Corollary \ref{cor_metric_entropy_vs_lambda_1}] \label{thmintro_example_small_metric_entropy}
 Let $\eps>0$ and $f\colon \R \rightarrow \R^{+}$ be any real continuous function such that $f(0)=0$, then there exists a negatively curved metric $g$ such that 
 \[
  f\left(h_{\mu}(g)\right) < \wt\lambda_1(g),  \text{ and } h_{\mu}(g)< \eps.
 \]
Furthermore, we can choose $g\in [\sigma]^{<}_A$, where $\sigma$ is a hyperbolic metric inside a fixed compact set of the Teichm\"uller space.
\end{thml}

In terms of the flexibility program, Theorem \ref{example_metric_entropy} shows that one can stay in the conformal classes of hyperbolic metrics in a fixed compact of the Teichm\"uller space (or any small neighborhood of at least one particular hyperbolic metric), and still obtain all the possible values for the metric entropy. We believe however that the stronger result below should be true.
\begin{conj}\label{conj_metric_entropy_flexible}
Let $(M,\sigma)$ be a hyperbolic surface of genus $\geqslant 2$ and total area $A$. Then, 
\[
\inf_{g\in[\sigma]_A^{<}}(h_{\mu}(g))=0.
\]
\end{conj}

Another goal of our study was to determine whether the topological entropy is as flexible as the metric entropy in a conformal class. It turns out that the flexibility of the topological entropy is linked to the flexibility of two other invariants: the linear drift $l(g)$ and the systole $\sys(g)$, i.e., the length of the shortest geodesic.

On one side, the topological entropy is, in negative curvature, always bounded above by a constant depending only on the area and the (inverse of the) systole (see Theorem \ref{bound_entr_systole}). So in particular, if $h_{top}(g)$ is unbounded, then $\sys(g)$ must go to zero.

%K.~Burns, in a personal communication, proved that the topological entropy is, in negative curvature, always bounded above by a constant depending only on the area and the (inverse of the) systole (see Theorem \ref{bound_entr_systole}). So in particular, if $h_{top}(g)$ is unbounded, then $\sys(g)$ must go to zero.

Conversely, Besson, Courtois and Gallot \cite[Corollaire 0.6]{BCG_margulis}, proved that if a family of negatively curved metrics with bounded diameter (this is always satisfied in our case due to Theorem \ref{thm_diam_no_conjugate}) is such that if $\sys(g)$ goes to zero, then $h_{top}(g)$ must be unbounded (this result also follows from the arguments of \cite[Section 2.3]{EK}). Finally, the variational principle and equation \eqref{relation_drifts_conform} implies that, if $l(g)$ goes to zero in a conformal class, then $h_{top}(g)$ must be unbounded (and hence, $\sys(g)$ must go to zero).

In this article, we obtained some partial results regarding the systole. We did not manage to show that the systole stays bounded away from zero in a class $[\sigma]_A^{\leqslant}$. However, we did prove that the ``obvious'' way of building a family of metric with systole going to zero cannot be done in a conformal class without positive curvature. More precisely, writing $l_g(\gamma)$ for the $g$-length of a curve $\gamma$, we show
\begin{thml}\label{thmintro_geodesics_do_not_shrink}
Let $(M,\sigma)$ be a hyperbolic surface of genus $\geqslant 2$ and total area $A$. For any $N>0$, and any closed $\sigma$-geodesic $\gamma$, such that $l_{\sigma}(\gamma) \leqslant M$, there exists a positive constant $\eps=\eps(N)$ such that, for any $g\in [\sigma]_A^{\leqslant}$, we have 
\[
 l_g(\gamma) \geqslant  \eps.
\]
\end{thml}
This result is presumably far from optimal, since our lower bound goes to zero as the length of the original $\sigma$-geodesic goes to infinity, and it seems at the very least counter-intuitive that one could shrink a very long curve without shrinking short ones. Although we did not pursue it, one could adapt our arguments to show that for any topologically non trivial closed curve $\gamma$ of controlled $\sigma$-length (i.e., bounded above), and controlled $\sigma$-geodesic curvature (both above and below), one can obtain a positive lower bound for the $g$-length of $\gamma$, with $g$ in $[\sigma]_A^{\leqslant}$. But again, the bound obtain via our technique will go to zero as the $\sigma$-length or $\sigma$-geodesic curvature goes to infinity. Nonetheless, we feel confident that the following should be true.
\begin{conj} \label{conj_systole_bounded_below}
Let $(M,\sigma)$ be a hyperbolic surface of genus $\geqslant 2$ and total area $A$. Then, there exists a constant $C = C(\sigma)>1$ such that, for any $g\in[\sigma]_A^{\leqslant}$, $\sys(g)>C^{-1}$, $h_{top}(g) <C$, and $C^{-1}< l(g)< C$.
\end{conj}

Finally, in Section~\ref{examples} we build examples, some folkloric, others new, showing that the conditions of conformality and negative curvature are essential for all our established and conjectural bounds.

We will use the following notations. The pair $(M, g)$ denotes a smooth closed Riemannian surface with metric $g$, and $v_g$ is the associated Riemannian measure on $M$. The universal cover of $M$ is denoted by $\wt M$, $\wt g$ is the lifted metric, and $d_{\wt M}(\cdot,\cdot)$ is the associated distance function.
We write $\Delta_g = \mathrm{div}_{g}(\grad_{g} )$ for the Laplacian of $g$ and $0 = \lambda_0(g)< \lambda_1(g)\leqslant \lambda_2(g)\leqslant \hdots \rightarrow \infty$ for its spectrum.
Throughout the text, since we will often have to switch between objects defined for different metrics, we write $g$-length, $g$-ball, $g$-geodesic, $g$-area, etc., to refer to the length, ball, geodesic, or area defined by the metric $g$.
Finally, we will also sometimes abuse terminology and refer to a shortest closed geodesic as ``a systole''.

\section{Restrictions for metrics with fixed total area in a fixed conformal class} \label{sec_restrictions}

In this section, we will prove Theorems \ref{thmintro_upper_bound_on_u} and \ref{thmintro_geodesics_do_not_shrink}, as well as other bounds that we can obtain from our conditions. The proofs in subsections \ref{subsec_upper_bound_conformal} and \ref{bounds_curves} follow essentially one overarching easy idea that we sketch now.

Suppose $\sigma$ is a fixed hyperbolic metric and $g=e^{2u}\sigma$ is a metric in its conformal class. Then, the Gaussian curvature $K_g$ of $g$ satisfy to the following
\[
 \Delta_{\sigma} u + 1 + K_g e^{2u} = 0.
\]
In particular, $g$ has nonpositive curvature if and only if 
\begin{equation}\label{eq_nonpositive_curvature}
 \Delta_{\sigma} u \geqslant -1.
\end{equation}
Now, for a function of two variables $u$, equation \eqref{eq_nonpositive_curvature} still allows a lot of flexibility. However, if $u$ was one-dimensional, then equation \eqref{eq_nonpositive_curvature} becomes very stringent. So, in the arguments, we average $u$ over some closed curves in local charts to obtain a one-dimensional function that satisfies equation \eqref{eq_nonpositive_curvature} and then leverage the condition that the $g$-area stays constant to get our bounds.

\subsection{Upper bound on the conformal factor} \label{subsec_upper_bound_conformal}

\begin{prop}\label{prop_control_of_u_on_circles_at_max}
 Let $(M,\sigma)$ be a hyperbolic surface of genus $\geqslant 2$ and total area $A$. Consider a smooth function $u\colon M \rightarrow \R$ such that $g= e^{2u}\sigma \in [\sigma]_A^{\leqslant}$. Let $\wt u$ be the lift of $u$ to $\wt M$ and $x_{max}\in \wt M$ be a point, where a global maximum of $\wt u$ is achieved. Then,
\[
 \int_{S_{\sigma}(x_{max},R)} \wt u dl_\sigma \geqslant  2\pi\sinh R \left(\max_{M}u - 2\log\left(\cosh\frac{R}{2}\right)\right),
\]
where $S_{\sigma}(x_{max},R)$ is the $\sigma$-sphere of radius $R$ centered at $x_{max}$.
\end{prop}

\begin{proof}
Passing to the universal cover, let $(r,\theta)$ be hyperbolic polar coordinates on the $\sigma$-ball $B_{\sigma}(x_{max},\rho) \subset\wt M$.
Recall that the Laplacian in polar coordinates is given by
\[
 \Delta_{\sigma} = \frac{\partial^2}{\partial r^2} + \frac{\cosh r}{\sinh r} \frac{\partial}{\partial r} + \frac{1}{\sinh^2 r} \frac{\partial^2}{\partial \theta^2}.
\]
Writing $\wt u$ for the lift of $u$ to the universal cover $\wt M$, we can use Green's theorem to get
\begin{align}\label{Green_thm_at_max}
\int_{B_{\sigma}(x_{max},\rho)}\Delta_{\sigma}\wt u dv_{\sigma} = \int_0^{2\pi}\frac{\partial u}{\partial r}(\rho,\theta)\sinh \rho \, d\theta.
\end{align}
The fact that $g$ has nonpositive curvature is equivalent to $\Delta_{\sigma}\wt u \geqslant -1$. Therefore, we have
\begin{align}\label{neg_curv_at_max}
\int_{B_{\sigma}(x_{max},\rho)}\Delta_{\sigma}\wt u dv_{\sigma}\geqslant -\vol(B_{\sigma}(x_{max},\rho)) = -2\pi(\cosh \rho - 1).
\end{align}
As a result, by (\ref{Green_thm_at_max}) and (\ref{neg_curv_at_max}) we obtain
\begin{align*}
\int_0^{2\pi}\frac{\partial u}{\partial r}(\rho,\theta)\, d\theta\geqslant -2\pi\frac{\cosh \rho - 1}{\sinh \rho}.
\end{align*}
Integrating the above inequality with respect to $\rho$ from $0$ to $R$ gives the proposition:
\begin{equation*}
 \int_0^R \int_0^{2\pi}\frac{\partial u}{\partial r}(\rho,\theta)\, d\theta d\rho = \int_0^{2\pi}u(R,\theta)\, d\theta -2\pi\max_{M}u \geqslant  - 4\pi\log\left(\cosh\frac{R}{2}\right). \qedhere
\end{equation*}
% We change the order of the differentiation with respect to $r$ and the integration with respect to $\theta$ in (\ref{ineq_at_max}). Then, we integrate with respect to $\rho$ from $0$ to $R$, and the proposition follows.
% \begin{align}
% &\frac{\partial}{\partial r}\left.\int_0^{2\pi}u(r,\theta)\, d\theta\,\right|_{r=\rho} \geqslant -2\pi\frac{\cosh \rho - 1}{\sinh \rho}\\
% &\int_0^{2\pi}u(R,\theta)\, d\theta \geqslant 2\pi\max_{M}u - 4\pi\log\left(\cosh\frac{R}{2}\right). 
% \end{align}
\end{proof}

\begin{thm}\label{bound_average_below_max_above}
Let $(M,\sigma)$ be a hyperbolic surface of genus $\geqslant 2$ and total area $A$. Then, there exists a positive constant $C = C(\sigma)$ such that for any smooth function $u\colon M \rightarrow \R$ such that $g= e^{2u}\sigma\in [\sigma]_A^{\leqslant}$ we have $u(x)\leqslant C$ for any $x\in M$.
\end{thm}

\begin{rem}
This result can be seen as a slight generalization of Schwarz lemma (see, for instance, \cite{Y}), where we do not assume that the curvature is bounded above by a negative constant.
% applied to two smooth compact Riemann surfaces $(M,\sigma)$ and $(M, g)$, where $\sigma$ and $g$ are conformally equivalent metrics. In Yau's theorem, it is assumed that $g$ has curvature bounded above by a negative constant.
\end{rem}

\begin{proof}
Again passing to the universal cover, let $(r,\theta)$ be hyperbolic polar coordinates on the $\wt\sigma$-ball $B_{\wt\sigma}(x_{max},\inj(\sigma)) \subset\wt M$, where $\inj(\sigma)$ is the injectivity radius of $(M,\sigma)$, and $x_{max}\in \wt M$ is a point where the global maximum of the lift $\wt u$ of $u$ is achieved.

By Proposition \ref{prop_control_of_u_on_circles_at_max} and Jensen's inequality, for every $0<R\leqslant \inj(\sigma)$ we have
\begin{align*}
\log\left(\int_0^{2\pi}e^{2u(R,\theta)}\,d\theta\right)\geqslant \log (2\pi) + \frac{1}{2\pi}\int_0^{2\pi}2u(R,\theta)\,d\theta\geqslant \log(2\pi)+2\max_M u - 4\log\left(\cosh\frac{R}{2}\right).
\end{align*}
Therefore, we obtain
\begin{equation*}
\int_0^{2\pi}e^{2u(R,\theta)}\,d\theta\geqslant 2\pi\frac{\exp(2\max_M u)}{\cosh^4\frac{R}{2}}.
\end{equation*}

The $g$-area of $B(x_{max}, \inj(\sigma))$ is bounded above by $A$, since the projection to $M$ restricts to a bijection on $B(x_{max}, \inj(\sigma))$. So, using the above inequalities, we obtain
\begin{multline*}
A \geqslant \int_0^{\inj(\sigma)}\int_0^{2\pi}e^{2u(r,\theta)}\sinh r\,d\theta dr \geqslant 2\pi\exp(2\max_M u)\int_0^{\inj(\sigma)}\frac{\sinh r}{\cosh^4\frac{r}{2}}\,dr \\
 = 2\pi\exp(2\max_M u)\int_0^{\inj(\sigma)}\frac{2\sinh \frac{r}{2}\cosh\frac{r}{2}}{\left(\cosh^2\frac{r}{2}\right)^2}\,dr = 4\pi\exp(2\max_M u)\left(1-\frac{1}{\cosh^2\frac{\inj(\sigma)}{2}}\right).
\end{multline*}

Therefore,

\begin{equation*}
\max_M u \leqslant \frac{1}{2}\log\frac{A}{4\pi\tanh^2\frac{\inj(\sigma)}{2}}. \qedhere
\end{equation*}
\end{proof}

\begin{cor}\label{bound_below_on_open_set}
Let $(M,\sigma)$ be a hyperbolic surface of genus $\geqslant 2$ and total area $A$. Then, for any integer $N\geqslant 1$ there exists a positive constant $K_1 = K_1(N, \diam(M))$ such that, for any open set $O\subset\wt M$ with $\diam(O)\leqslant N\diam(M)$ and any smooth function $u\colon M \rightarrow \R$ such that $g= e^{2u}\sigma\in [\sigma]_A^{\leqslant}$, we have
\begin{equation*}
\iint_{O}\wt udv_{\sigma}\geqslant -K_1,
\end{equation*}
where $\wt u$ is the lift of $u$ to the universal cover.
\end{cor}

\begin{proof}
Given our assumptions on $O$, there exists a $\sigma$-ball $B(x_{max},R)$ centered at a point $x_{max}$, where $R = (N+1)\diam(M)$ the global maximum of $\wt u$ is achieved, such that $O\subset B(x_{max},R)$. 

Then,
\begin{align*}
\iint_{O}\wt udv_{\sigma}= \iint_{B(x_{max},R)}\wt udv_{\sigma} - \iint_{B(x_{max},R)\smallsetminus O}\wt udv_{\sigma}.
\end{align*}
Also, we notice that $\max_M u\geqslant 0$ as $g$ has the same area as $\sigma$. In the polar coordinates $(r,\theta)$ for the metric $\sigma$ in the ball $B(x_{max}, R)$, by Proposition \ref{prop_control_of_u_on_circles_at_max} we have
\begin{align*}
&\iint_{B(x_{max}, R)}\wt udv_{\sigma} = \int_0^{R}\int_0^{2\pi}\wt u\sinh r\, d\theta dr\geqslant -4\pi\int_0^{R}\sinh r\log\left(\cosh\frac{r}{2}\right)\,dr = \nonumber\\ &= -2\pi-4\pi\left((\cosh R+1)\log\left(\cosh\frac{R}{2}\right)-\frac{1}{2}\cosh R\right).
\end{align*}

From the inequality above and Theorem \ref{bound_average_below_max_above}, it is easy to deduce the Corollary.
\end{proof}

Thanks to Theorem \ref{bound_average_below_max_above}, we can bound the Laplace spectrum of nonpositively curved metrics in a given conformal class.

Recall that $0=\lambda_0(g)< \lambda_1(g) \leq \lambda_2(g) \leq \dots $ denotes the Laplace spectrum of the metric $g$ on $M$,
and $\wt\lambda_1(\wt g)$ is the bottom of the $L^2$-spectrum of the lifted Laplacian on the universal cover $\wt M$.

The Min-Max principle state that, for any $k \in \N$, 
\begin{equation}\label{min-max}
\lambda_k(g) = \inf_{V_k}\sup\{R_g(f)| f\neq 0, f\in V_k\},
\end{equation}
where $V_k$ runs through all the $k+1$-dimensional subspaces of the Sobolev space $H^1(M)$ and $R_g(f)$ is the Rayleigh quotient of $f$, i.e.,
\begin{align*}%\label{Rayleigh_quotient}
R_g(f) = \frac{\int\limits_{M}|\nabla_g f|_g^2\,dv_g}{\int\limits_{M}f^2\,dv_g}.
\end{align*} 

Note that, when $g = e^{2u}\sigma$, the Rayleigh quotient becomes
\begin{align*}
R_g(f) = \frac{\int\limits_{M}|\nabla_{\sigma} f|_{\sigma}^2\,dv_{\sigma}}{\int\limits_{M}e^{2u}f^2\,dv_{\sigma}}.
\end{align*}

Hence, using the Min-Max principle, we immediately obtain the following corollary from Theorem~\ref{bound_average_below_max_above}.
\begin{cor}\label{bound_Laplace_spectrum_conformal}
Let $(M,\sigma)$ be a hyperbolic surface of genus $\geq2$ and total area $A$.

There exists a constant $C = C(\sigma)>0$ such that, for any Riemannian metric $g\in [\sigma]^{\leqslant}_A$, and all $k \in \mathbb{N}$,
\[
\lambda_k(g)\geqslant C\lambda_k(\sigma). 
\]

In particular, for any nonpositively curved metric in a fixed conformal class, the function $\lambda_1(g) \vol_g(M)$ is uniformly bounded away from $0$.

Similarly, on the universal cover, we have 
\[
 \wt \lambda_1(\wt g) \geqslant C \wt \lambda_1(\wt \sigma).
\]
\end{cor}

\subsection{Lower bounds on the length of some curves}\label{bounds_curves}

In this section, we prove lower bounds for the integral of $u$ on a number of special curves. This results can be translated using Jensen's inequality to lower bounds on the $g$-length of these curves. In particular, one gets Theorem \ref{thmintro_geodesics_do_not_shrink} from Proposition \ref{prop_bounded_length_of_hyperbolic_geodesic} in that way.

\begin{prop}\label{prop_control_of_u_on_circles}
 Let $(M,\sigma)$ be a hyperbolic surface of genus $\geqslant 2$ and total area $A$. Let $x\in \wt M$ and $0<R_1\leq R_2$ arbitrary. There exists a positive constant $K_2= K_2(\sigma, R_1, R_2)$ such that for any smooth function $u\colon M \rightarrow \R$ such that $g= e^{2u}\sigma\in [\sigma]_A^{\leqslant}$ and any $r\in [R_1,R_2]$, we have 
\[
 \int_{S_{\sigma}(x,r)} \wt u dl_\sigma \geqslant - K_2,
\]
where $S_{\sigma}(x,r)$ is the $\sigma$-sphere of radius $r$ centered at $x$ and $\wt u$ is the lift of $u$ to the universal cover.
\end{prop}

\begin{proof}
 Let $(r,\theta)$ be hyperbolic polar coordinates on the $\sigma$-ball $B_{\sigma}(x,R_2) \subset\wt M$.
%Recall that the Laplacian in polar coordinates is given by
%\[
% \Delta = \frac{\partial^2}{\partial r^2} + \frac{\cosh r}{\sinh r} \frac{\partial}{\partial r} + \frac{1}{\sinh^2 r} \frac{\partial^2}{\partial \theta^2}.
%\]

Let $0\leq \eps \leqslant \rho$ be arbitrary. Let $A_{\sigma}(\eps,\rho) = \{ (r,\theta) \mid \eps \leqslant r \leqslant \rho\}$. There exists an integer $N(\rho)\geqslant 1$ depending on $\rho$ and the hyperbolic metric $\sigma$ such that the annulus $A_{\sigma}(\eps,\rho)$ is contained in at most $N(\rho)$ fundamental domains for $M$ in $\wt M$. Call this minimal union of fundamental domains $F$. Since $\Delta u \geqslant -1$ and 
\[
 \int_F \Delta \wt u dv_{\sigma} = N(\rho) \int_M \Delta u dv_{\sigma}= 0,
\]
we get that 
\[
 \int_{A_{\sigma}(\eps,\rho)} \Delta \wt u dv_{\sigma} = - \int_{F\smallsetminus A_{\sigma}(\eps,\rho)} \Delta \wt u dv_{\sigma} \leqslant \int_{F\smallsetminus A_{\sigma}(\eps,\rho)} dv_{\sigma} \leqslant  AN(\rho).
\]

Now, 
\begin{align*}
 \int_{A_{\sigma}(\eps,\rho)} \Delta \wt u dv_{\sigma} &= \int_0^{2\pi}\int_{\eps}^{\rho} \Delta \wt u \sinh r dr d\theta \\
 &= \int_0^{2\pi} \sinh \rho \frac{\partial u (\rho, \theta)}{\partial r} - \sinh \eps \frac{\partial u (\eps, \theta)}{\partial r} d\theta.
\end{align*}

Hence, taking $\eps=0$ and changing the order of the integration and differentiation, we obtain
\begin{equation*} 
 \frac{\partial }{\partial r} \left. \int_0^{2\pi} u(r, \theta) d\theta\,\right|_{r=\rho} \leqslant \frac{AN(\rho)}{\sinh(\rho)}.
\end{equation*}

Since we also have that $\int_{A_{\sigma}(\eps,\rho)} \Delta \wt u dv_{\sigma} \geq - \vol(A_{\sigma}(\eps,\rho)) $, we deduce

\begin{equation*}
 \frac{\partial }{\partial r} \left. \int_0^{2\pi} u(r, \theta) d\theta\,\right|_{r=\rho} \geqslant - \frac{\vol(A_{\sigma}(0,\rho))}{\sinh(\rho)}.
\end{equation*}

From the control of derivatives given by the two equations above and Corollary \ref{bound_below_on_open_set}, it is easy to deduce the Proposition. 
\end{proof}

\begin{cor}
Let $(M,\sigma)$ be a hyperbolic surface of genus $\geqslant 2$ and total area $A$. Let $0<R_1\leqslant R_2$ be arbitrary. There exists a positive constant $K_3= K_3(\sigma, R_1,R_2)$ such that for any smooth function $u\colon M \rightarrow \R$ such that $g= e^{2u}\sigma\in [\sigma]_A^{\leqslant}$, any $x\in \wt M$ and any $r\in [R_1,R_2]$, we have 
\[
 \int_{L_r} \wt u dl_\sigma \geqslant - K_3,
\]
where $L_r\subset S_{\sigma}(x,r)$ is a piece (a union of pieces) of arcs of the $\sigma$-sphere $S_{\sigma}(x,r)$ of radius $r$ centered at $x$ and $\wt u$ is the lift of $u$ to the universal cover.
\end{cor}

\begin{proof}
The result follows from Proposition \ref{prop_control_of_u_on_circles} and Theorem \ref{bound_average_below_max_above}.
\begin{equation*}
\int_{L_r} \wt u dl_\sigma = \int_{S_{\sigma}(x,r)} \wt u dl_\sigma - \int_{S_{\sigma}(x,r)\smallsetminus L_r} \wt u dl_\sigma\geqslant -K_2-Cl_{\sigma}(S_{\sigma}(x,r)\smallsetminus L_r)
\end{equation*}
\end{proof}

\begin{prop}\label{prop_bounded_length_of_hyperbolic_geodesic}
Let $(M,\sigma)$ be a hyperbolic surface of genus $\geqslant 2$ and total area $A$. For any closed $\sigma$-geodesic $\gamma$, there exists a positive constant $K_4 = K_4(l_\sigma(\gamma))$ such that, for any smooth function $u\colon M \rightarrow \R$ such that $g= e^{2u}\sigma\in [\sigma]_A^{\leqslant}$, we have 
\[
 \int_{\gamma} u dl_\sigma \geq - K_4.
\]
\end{prop}

\begin{rem}
Our proof gives a counter-intuitive dependency of $K_4=K_4(l_\sigma(\gamma))$ in the geodesic $\gamma$. Indeed, due to the use of Corollary \ref{bound_below_on_open_set}, $K_4(l_\sigma(\gamma))$ goes to infinity as $l_\sigma(\gamma)$ goes to infinity.
One could presumably obtain a uniformly bounded constant by being a bit more careful, but we choose not to pursue that direction as it does not lead to any significant improvement regarding Conjecture~\ref{conj_systole_bounded_below}.

 Note also that the proof shows that for \emph{any} curve $\alpha$ that is obtained as an equidistant curve from a $\sigma$-geodesic $\gamma$, we have $ \int_{\alpha} u dl_\sigma \geq - K_4$.
\end{rem}

\begin{proof}
Let $M_c$ be a cylindrical cover of $M$ with fundamental group generated by the homotopy class of $\gamma$. Consider hyperbolic normal polar coordinates $(r,\theta)$ on $M_c$, where $\gamma$ is described by the equation $r=0$ and $\theta\in[0;l_{\sigma}(\gamma)]$. Let $u_c$ be the lift of $u$ to $M_c$. 

Let $0\leq \eps \leq \rho$ be arbitrary and $A_{\sigma}(-\rho, -\eps) = \{ (r,\theta) \mid -\rho \leq r \leq -\eps\}$.

Green's theorem implies that
\begin{align}\label{Green_collar}
\iint_{A_{\sigma}(-\rho, -\eps)}\Delta_{\sigma}u_c dv_{\sigma} &= \int_0^{l_{\sigma}(\gamma)}\frac{\partial u_c}{\partial r}(-\eps,\theta)\cosh \eps\, d\theta - \int_0^{l_{\sigma}(\gamma)}\frac{\partial u_c}{\partial r}(-\rho,\theta)\cosh \rho\, d\theta \\
 & = \cosh \eps \frac{\partial}{\partial r}\left. \int_0^{l_{\sigma}(\gamma)} u_c(r,\theta)\,d\theta\,\right|_{r=-\eps} - \cosh \rho \frac{\partial}{\partial r} \left.\int_0^{l_{\sigma}(\gamma)} u_c(r,\theta)\,d\theta\,\right|_{r=-\rho}.\nonumber
\end{align}

Recall that $g$ being nonpositively curved implies that $\Delta_{\sigma}u\geqslant -1$, so $\Delta_{\sigma}u_c\geqslant -1$. Therefore, 
\begin{align*}
\iint_{A_{\sigma}(-\rho, -\eps)}\Delta_{\sigma}u_c dv_{\sigma}\geqslant -\vol(A_{\sigma}(-\rho, -\eps)) = -(\sinh \rho - \sinh \eps)l_{\sigma}(\gamma)
\end{align*}

Now, by Corollary \ref{bound_below_on_open_set} applied to $A_{\sigma}(-\diam(M),0)$ and the equality 
$$
\iint_{A_{\sigma}(-\diam(M), 0)}u_c\,dv_{\sigma} = \int_0^{\diam(M)}\int_0^{l_\sigma(\gamma)}u \cosh r\,d\theta dr,
$$
there exists a positive constant $K_5 = K_5(\diam(M), l_{\sigma}(\gamma),\sigma)$ and $R_1\in[0;\diam(M)]$ such that 
\begin{equation*}
\int_0^{l_{\sigma}(\gamma)}u_c(-R_1,\theta)\, d\theta\geqslant -\frac{K_5}{\diam(M)}.
\end{equation*}

\begin{rem}
The constant $K_5$ goes to infinity as the length of $\gamma$ goes to infinity due to the use of Corollary \ref{bound_below_on_open_set}.
\end{rem}

If $R_1=0$, then the Proposition is proven. 
So we suppose that $R_1\neq 0$. Then, there exists a constant $R_2\geqslant R_1$ and smaller than some constant $R = R(\diam(M), K_5, \sigma)$ such that 
\begin{align*}
\int_0^{l_{\sigma}(\gamma)}u_c(-R_2,\theta)\, d\theta\geqslant -\frac{K_5}{\diam(M)},\\
\frac{\partial}{\partial r}\left.\int_0^{l_{\sigma}(\gamma)}u_c(r,\theta)\, d\theta \, \right|_{r=-R_2}\geqslant -K_5.
\end{align*}

Indeed, either $R_1$ itself works, or
\[
\frac{\partial}{\partial r}\left.\int_0^{l_{\sigma}(\gamma)}u_c(r,\theta)\, d\theta \, \right|_{r=-R_1}< -K_5 \leq 0. 
\]
So, at least locally around $-R_1$, the function $\int_0^{l_{\sigma}(\gamma)}u_c(r,\theta)\, d\theta$ is strictly decreasing. 
Then, for all $r \geq R_1$, sufficiently close to $R_1$, we have
\[
 \int_0^{l_{\sigma}(\gamma)}u_c(-r,\theta)\, d\theta> \int_0^{l_{\sigma}(\gamma)}u_c(-R_1,\theta)\, d\theta \geqslant  -\frac{K_5}{\diam(M)}.
\]
However, since $u$ is uniformly bounded above (by Theorem \ref{bound_average_below_max_above}), for all $r$ we have, 
\[ 
\int_0^{l_{\sigma}(\gamma)}u_c(-r,\theta)\, d\theta\leqslant Cl_{\sigma}(\gamma).
\] 
 Hence we cannot have
\[
\frac{\partial}{\partial r}\left.\int_0^{l_{\sigma}(\gamma)}u_c(r,\theta)\, d\theta \, \right|_{r}< -K_5, 
\]
for all $r\geq R_1$. Thus $R_2$ exists as claimed.

Now take $\rho = R_2$ in equation (\ref{Green_collar}) and recall that $R_2\leqslant R$. Assume that for any $D>0$ there exists $u$ such that $\int_{\gamma} u dl_\sigma< -D$. Then, there exists $\eps\in[0;R_2]$ such that 
\begin{equation*}
\frac{\partial}{\partial r} \left.\int_0^{l_{\sigma}(\gamma)} u_c(r,\theta)\,d\theta\,\right|_{r=-\eps}\leqslant -\frac{D-\frac{K_5}{\diam(M)}}{R}.
\end{equation*}   
 
The constant $D$ can be chosen arbitrary large. As a result, we obtain a contradiction as the left-hand side of the equality (\ref{Green_collar}) is bounded below and the right-hand side of it can be arbitrary small. Therefore, the Proposition follows.
\end{proof}

Using Jensen's inequality we can restate all the results of this section, replacing the integral of $u$ over the curves by the $g$-length of the curves. In particular, Theorem \ref{thmintro_geodesics_do_not_shrink} is a direct corollary of Proposition \ref{prop_bounded_length_of_hyperbolic_geodesic}.

\subsection{Upper bound on diameter for metrics without conjugate points}

As mentioned in the introduction, a direct consequence of the upper bound on the conformal factor given by Theorem \ref{bound_average_below_max_above} is that the diameter stays bounded in a class $[\sigma]_A^{\leqslant}$. But we can further relax our assumptions and still get a bound on the diameter.
\begin{thm}\label{thm_diam_no_conjugate}
Let $(M,\sigma)$ be a compact hyperbolic surface of genus $\geqslant2$ and total area $A$. Then, there exists a constant $D= D(\sigma)>0$ such that for any smooth metric $g$ without conjugate points and total area $A$ that is conformally equivalent to $\sigma$, we have 

\[
\mathrm{diam}(M,g)\leqslant D. 
\]

\end{thm}

The proof of the Theorem above relies on the following lemma, which says that one cannot increase infinitely the length inside a fixed disc without creating conjugate points. We state the lemma in a more general context as the one needed for Theorem \ref{thm_diam_no_conjugate}, as we will also be using that lemma in the proof of Theorem \ref{example_metric_entropy}.

\begin{lem}\label{lem_diameter_no_conjugate_point}

Let $(M,h)$ be a nonpositively curved surface of genus $\geqslant 2$ and total area $A$. Let $\eps>0$ be arbitrary. Suppose that for some $x\in M$, the metric $h$ is invariant under rotations in the $h$-ball $B_{h}(x,2\eps) \subset M$. Then, there exists a constant $C>0$, depending only on $\eps$ and the metric $h$ such that, for any Riemannian metric $g= e^{2u}h$ with no conjugate points and total area $A$, we have, for all $y,z\in B_{h}(x,\eps)$,
\[
  d_g(y,z) \leq C.
\]
% Let $(M,\sigma)$ be a hyperbolic surface of genus $\geq2$ and total area $A = 2\pi|\chi(M)|$, where $\chi(M)$ is the Euler characteristic.
% Let $\eps>0$ be arbitrary. Consider any $\sigma$-ball $B_{\sigma}(x,\eps) \subset M$. There exists a constant $C>0$, depending only on $\eps$ and the metric $\sigma$ such that the following holds:
% For any $u\colon M \rightarrow \R$ smooth function such that the Riemannian metric $g= e^{2u}\sigma$ has no conjugate points and has total area $A$, then for all $y,z\in B_{\sigma}(x,\eps)$,
% \[
%  d_g(y,z) \leq C.
% \]
\end{lem}

\begin{rem}\label{rem_diameter_no_conjugate_point}
 As we will see in the proof, the constant $C$ actually depends only on $\eps$ and an upper bound for the function $\sqrt{\det h}$ in $B_{h}(x,2\eps)$
\end{rem}

\begin{proof}
 Let $(r, \theta)$ be polar coordinates in the ball $B_{h}(x,2\eps)$. 

The $g$-area of an annulus $T(\eps) = \{(r,\theta)| \eps\leqslant r\leqslant 2\eps, 0\leqslant\theta\leqslant 2\pi \}$ is smaller than $A$ and is equal to
\begin{align*}
\iint_{T(\eps)}e^{2u}dv_{\sigma} = \int_{\eps}^{2\eps}\int_{0}^{2\pi}e^{2u(c_{\theta}(r))}f(r)\,d\theta dr,
\end{align*}
where $f(r) = \sqrt{\det g}$, $\det g$ is the determinant of the metric tensor and $c_\theta(r)$ is a $h$-geodesic radius corresponding to the angle $\theta$. The function $\sqrt{\det g}$ depends only on $r$ in $B_{h}(x,2\eps)$ because $h$ is assumed to be invariant under rotations in that ball.

Therefore, there exists $\bar r\in[\eps; 2\eps]$ such that 
\begin{equation*}
\int_{0}^{2\pi}e^{2u(c_{\theta}(\bar r))}f(\bar {r})\,d\theta\leqslant \frac{A}{\eps}.
\end{equation*}
It follows that
\begin{equation*}
\int_{0}^{2\pi}e^{2u(c_{\theta}(\bar r))}d\theta\leqslant \frac{A}{\eps f(\bar r)}.
\end{equation*}
Moreover, if we denote $C_{\bar r} = \{(r, \theta) \mid r=\bar r,0\leqslant \theta< 2\pi\}$, then the $g$-length of $C_{\bar r}$ is controlled:
\begin{equation*}
 l_g(C_{\bar r}) = \int_{0}^{2\pi}e^{u(c_{\theta}(\bar r))}f(\bar r)\,d\theta\leqslant f(\bar r)\sqrt{2\pi}\left(\int_{0}^{2\pi}e^{2u(c_{\theta}(\bar r))}d\theta\right)^{\frac{1}{2}} \leqslant  \sqrt{\max_{r\in[\eps;2\eps]} f(r)}\left(\frac{2\pi A}{\eps}\right)^{\frac{1}{2}}=:W.
\end{equation*}

As a result, any two points $p,q \in C_{\bar r}$ can be connected by two arcs of $C_{\bar r}$ of $g$-length less than or equal to $W$. Fix any two points $p,q \in C_{\bar r}$. Let $z \in B_h(x,\eps)$. We want to show that either $d_g(z,p) \leq W$, or $d_g(z,q) \leq W$. Once this is established, the lemma will follow with $C= 3W$.

So, suppose this is not the case, i.e., suppose that the $g$-length of any path connecting $p$ to $q$ and passing through $z$ is strictly greater than $W$. 
Let $\wt B_h(x,\eps)$ be a lift of $B_h(x,\eps)$ in the universal cover $\wt M$. Let $\wt p, \wt q, \wt z \in \wt B_h(x,\eps)$ the associated lifts of $p,q$ and $z$. By what we proved so far and our assumption, the space of curves connecting $\wt p$ to $\wt q$ of $g$-length less than or equal to $W$ has two connected components. Therefore, there exist two different geodesics connecting $\wt p$ and $\wt q$. This contradicts the fact that $g$ has no conjugate points. Hence, $z$ is at $g$-distance less than $W$ from either $p$ or $q$.
\end{proof}

The proof of Theorem \ref{thm_diam_no_conjugate} follows easily from Lemma \ref{lem_diameter_no_conjugate_point}

\begin{proof}[Proof of Theorem \ref{thm_diam_no_conjugate}]
Choose a minimal cover of $M$ by balls of radius $\inj(\sigma)$, the injectivity radius of $\sigma$. Let $N(\sigma)$ be the number of balls in the cover. Let $z_1,z_2\in M$. The $\sigma$-geodesic between $z_1$ and $z_2$ is covered by at most $N(\sigma)$ balls. Now, since $\sigma$ is invariant under rotations in every balls of the cover, we can apply Lemma \ref{lem_diameter_no_conjugate_point} with $h=\sigma$. It gives us the existence of a constant $C(\inj(\sigma))$ such that $d_g(z_1,z_2) \leq N(\sigma)C(\inj(\sigma))$.
\end{proof}

\section{Flexibility for negatively curved metrics with fixed total area in a fixed conformal class} \label{sec_flexibility}

In this section, we show that the metric entropy is flexible inside a conformal class and prove Theorem \ref{thmintro_example_small_metric_entropy}.

\begin{thm}\label{example_metric_entropy}
Consider a surface $M$ of genus $\geqslant 2$ and total area $A$. There exists a compact set $\mathcal K$ in the Teichm\"{u}ller space on $M$ on which the following holds. For any positive constant $\eps>0$ there exists a negatively curved Riemannian metric $g$ on $M$, such that:
\begin{enumerate}
 \item the total area of $M$ with respect to $g$ is $A$;
 \item the metric $g$ is conformally equivalent to a hyperbolic metric in $\mathcal K$;
 \item the metric entropy, $h_{\mu}(g)$, of $g$, i.e., the entropy of the geodesic flow of $g$ with respect to its Liouville measure, satisfy
 \[
  0< h_{\mu}(g) \leqslant \eps.
 \]
\end{enumerate}
\end{thm}

\begin{rem} %\label{rem_topological_entropy_close_to1}
 As in \cite{EK}, one might want to chose $g$ in the theorem above in such a way that its topological entropy $h_{top(g)}$ is as close to $\left(2\pi|\chi(M)|/A\right)^{1/2}$ as one wishes. However, if we try to do this following \cite{EK} here, then the compact set $\mathcal K$ that we obtain will, a priori, have to become infinitely big. In particular, we do not know whether small values of $h_{\mu}$ forces a gap for $h_{top}$ or not.
\end{rem}

Thanks to Theorem \ref{example_metric_entropy}, we immediately obtain the
\begin{cor}\label{arbitr_small_in_nbd}
Consider a surface $M$ of genus $\geqslant 2$ and $A>0$. There exists a hyperbolic metric $\sigma$ on $M$ of total area $A$ such that in any neighborhood $\mathcal U$ of $\sigma$ in the Teichmuller space for any $\eps>0$ there exists a hyperbolic metric $\sigma'\in \mathcal U$ of total area $A$ with $\inf\limits_{g\in[\sigma']_A^{<}}h_{\mu}(g)<\eps$. 
\end{cor}

\begin{rem}
Notice that one can apply the normalized Ricci flow to a metric  $g\in[\sigma']_A^{<}$ such that $h_{\mu}(g)<\eps$, and, using the continuity of the metric entropy (see \cite{KKPW}), get that all the values in $[\eps, \left(2\pi|\chi(M)|/A\right)^{1/2}]$ are realized as the metric entropy of some metric in $[\sigma']_A^{<}$.
\end{rem}

\begin{rem}
For $\sigma$ as in Corollary \ref{arbitr_small_in_nbd}, we actually expect that $\inf_{g\in[\sigma]_A^{<}}h_{\mu}(g)=0$. Indeed, for any $n$, one can pick $\sigma_n$ a hyperbolic metric and $g_n = e^{2u_n}\sigma_n$ such that $\sigma_n$ converges in the $C^{\infty}$ norm to $\sigma$, and $h_{\mu}(g_n)<1/n$, then, by continuity, one would expect that the metric entropy of $h_n = e^{2u_n}\sigma$ converges to $0$. There are however two big problems to try to make that argument works. First, $h_n$ has no reason a priori to be negatively curved, as the curvature of $g_n$ has to converge to zero. Second, even if we knew that $h_n$ was negatively curved, we also know, from the proof of Theorem~\ref{example_metric_entropy}, that the functions $u_n$ can be $C^0$, but not $C^1$ close as $n$ goes to infinity. Hence one cannot use the known results about continuity of the metric entropy.

In spite of this, we do believe that the metric entropy should be totally flexible in a conformal class. We even believe (Conjecture~\ref{conj_metric_entropy_flexible}) that $\inf_{g\in[\sigma]_A^{<}}h_{\mu}(g)=0$ for \emph{any} hyperbolic metric $\sigma$ on $M$ with total area $A$.
\end{rem}

\begin{cor}\label{cor_metric_entropy_vs_lambda_1}
 Let $\eps>0$ and $f\colon \R \rightarrow \R^{+}$ be any real continuous function such that $f(0)=0$, then there exists a negatively curved metric $g$ such that 
 \[
  f\left(h_{\mu}(g)\right) < \wt\lambda_1(g), \text{ and } h_{\mu}(g) < \eps.
 \]
Furthermore, we can choose $g\in [\sigma]^{<}_A$, where $\sigma$ is a hyperbolic metric inside a fixed compact set of the Teichm\"uller space.
\end{cor}

\begin{proof}
This follows directly from the combination of Corollary \ref{bound_Laplace_spectrum_conformal} and Theorem \ref{example_metric_entropy}.
\end{proof}

\begin{proof}[Proof of Theorem \ref{example_metric_entropy}]

In Section 3.1 of \cite{EK}, the second author and Katok builds a family of negatively curved metrics satisfying the conditions 1) and 3). We will show here that this family of example also satisfy the condition 2). To prove that, we show that the hyperbolic metrics which are in the same conformal class have diameter uniformly bounded above, and hence, stay in a compact part of the Teichm\"uller space. The proof of the uniform bound on the diameter follows the same line as the proof of Theorem \ref{thm_diam_no_conjugate}.

We start by a quick description of the construction given in \cite{EK} of the relevant metrics. 

Let $\sigma$ be a fixed hyperbolic metric on $M$ of total area $A$. Let $d>0$ be sufficiently small. 
Let $\mathcal T$ be a triangulation by $\sigma$-geodesic triangles of $M$ such that each edge is a geodesic segment of $\sigma$-length between $d/3$ and $d$. Such a triangulation exists: We can take any geodesic triangulation and refine it in such a way that it satisfies the condition. 

To each triangle $T$ in $\mathcal T$, consider its comparison triangle $T_\ast$ in Euclidean space, i.e., the Euclidean triangle $T_\ast$ has sides of Euclidean length equal to the $\sigma$-length of the corresponding side in $T$. From all the comparison triangles $T_\ast$, we construct a polyhedral surface (homeomorphic to $M$) with conical singularities by gluing them according to the incidence of the original triangles in $\cT$. In particular, we thus obtain a singular metric $g_0$ on $M$ with zero curvature everywhere except from the conical singularities at the vertices of $\cT$ where the angle is larger than $2\pi$. 

For sufficiently small $d$, the ratio of the area between a triangle $T \in \cT$ and its comparison triangle $T_\ast$ becomes arbitrarily close to $1$ (see \cite[Lemma 3.2]{EK}). Hence, we can choose the triangulation $\cT$ in such a way that the total area of $M$ for $g_0$ is comprised between, say, $\left(1-\frac{1}{1000}\right)A$ and $\left(1+\frac{1}{1000}\right)A$.

For all $\alpha>0$ sufficiently smaller than $d$, we can find a family of smooth Riemannian metrics $\{g_{\eps}\}$, $\eps\in(0,\alpha]$, obtained by smoothing of $g_0$, such that (see \cite{EK} for the details):
\begin{itemize}
 \item each metric $g_{\eps}$ has nonpositive curvature;
 \item the metric $g_{\eps}$ coincide with $g_0$ outside of some $\eps$-neighborhoods (for $g_0$) of the conical singularities of $g_0$;
 \item the total area of $M$ for $g_{\eps}$ is between $\left(1-\frac{1}{100}\right)A$ and smaller than $\left(1+\frac{1}{100}\right)A$.
\end{itemize}

Let $p$ be a conical point of $g_0$. For $\eps$ fixed, consider the polar coordinates $(r,\theta)$ centered at $p$ for the metric $g_{\eps}$. As usual, $\theta$ is the angular coordinate and $r$ is the $g_{\eps}$-distance to $p$. The smoothing procedure in Lemma 3.3 of \cite{EK} is such that a sufficiently small $g_{\eps}$-ball centered at the conical point of $g_0$ is invariant under rotations with respect to the center. Therefore, all such $g_0$-balls are also $g_\eps$-balls for any $\eps$ (but of varying radius). Moreover, the construction gives uniform upper and lower bounds on the $g_{\eps}$-radii of such balls in terms of $\alpha$ and $\cT$ (see \cite{EK} for details). Furthermore, since the metrics $g_\eps$ and $g_0$ coincide outside of neighborhoods of the conical points, the $g_0$-balls which do not intersect $\alpha$-neighborhoods of singular points of $g_0$ are also $g_{\eps}$-balls of the same radii.

Let $\{B_{g_0}(x_i,r_i)\}_{i=1,\dots, N}$ be an open cover of $M$ by $g_0$-balls such that, for every $i$, $B_{g_0}(x_i,2r_i)$ is either entirely contained in a flat part of $g_0$, or $x_i$ is a conical point of $g_0$ and $B_{g_0}(x_i,2r_i)$ does not contain any other conical points. We further assume that $\{B_{g_0}(x_i,r_i)\}_{i=1,\dots, N}$ is a minimal such cover, i.e., it has the smallest number of balls among all covers satisfying to our conditions. So $N$ is a number depending only on $g_0$, or equivalently, depending only on the topology of $M$, the starting hyperbolic metric $\sigma$ and the choice of $d$. Let $R=\max_i r_i$. Similarly, $R$ depends only on $g_0$.

We now choose $\alpha$ smaller if necessary, so that all the balls $B_{g_0}(x_i,2r_i)$ that are entirely contained in a flat part of $g_0$ are disjoint from every $\alpha$-neighborhood of the conical points of $g_0$.

Let $\wt C = \sup \{ \sqrt{\det g_\eps} \mid \eps\in(0,\alpha] \}$. Given the smoothing procedure used, $\wt C$ is finite and depends only on $g_0$, $\cT$, $\alpha$ and $R$ (this follows from formula (3.5) in \cite{EK}). 

Since, by construction, for $\eps\in(0,\alpha]$, all of the $g_{\eps}$ are rotationally invariant in each ball $B_{g_0}(x_i,r_i)$, we can apply Lemma~\ref{lem_diameter_no_conjugate_point} to all of them. Moreover, according to Remark \ref{rem_diameter_no_conjugate_point}, the constant $C$ given by the lemma depends only on $g_0$, $\cT$, $\wt C$, $\alpha$ and $R$. 

Thus, we deduce as in the proof of Theorem \ref{thm_diam_no_conjugate} that the diameter of any metric without conjugate points conformally equivalent to one of $g_{\eps}$, $\eps\in(0,\alpha]$, and of total area $A$ is bounded above by $NC$. In particular, for any $\eps$, the hyperbolic metric $\sigma_{\eps}$ of total area $A$, that is conformally equivalent to $g_{\eps}$, has diameter bounded by $N C$. Hence, for any $\eps$, the hyperbolic metrics $\sigma_{\eps}$ must stay in a compact subset of the Teichm\"{u}ller space.

It is shown in Section 3.3 of \cite{EK} that the metric entropy of $g_{\eps}$ tends to $0$ as $\eps$ tends to $0$. 
Now, by construction, the total area of each $g_{\eps}$ is bounded between $\left(1-\frac{1}{100}\right)A$ and $\left(1+\frac{1}{100}\right)A$. Therefore, the metrics $\bar g_{\eps} = g_{\eps} \frac{A}{\vol_{g_{\eps}}(M)}$ are smooth nonpositively curved metrics, of total area $A$, and their metric entropy still converges to $0$ with $\eps$.

Finally, to obtain metrics that are negatively curved, instead of just nonpositively, we can smoothly approximate the metrics $\bar g_{\eps}$ by metrics of negative curvature with total area $A$ that are conformally equivalent to $\bar g_{\eps}$ by essentially taking some negative curvature from the neighborhoods of the conical points and distribute it among regions where we have zero curvature (see Proposition 5.1 in \cite{EK} for more details). Another way to avoid nonpositive curvature is to replace initial triangles of zero curvature by triangles of constant negative curvature close to $0$ and repeat the construction.
\end{proof}

\section{Flexibility for metrics with fixed total area under weaker conditions}\label{examples}
In this section, we construct a number of examples, some relatively well-known, some potentially new, showing that none of the bounds obtained in Section \ref{sec_restrictions} still hold when we relax our hypothesis.
\subsection{Examples of metrics with arbitrary short systole}
\begin{lem}\label{Example_systole}
Let $M$ be a surface of genus $\geqslant 2$ equipped with a hyperbolic metric $\sigma$ of total area $A$. Then, for every $\eps>0$ there exists a Riemannian metric $g$ on $M$ of total area $A$ that is conformally equivalent to $\sigma$ such that $\sys(g) \leq \eps$ and $\diam(M,g) \leq 2\diam(M,\sigma)$.
\end{lem}

Notice that the metric we construct for this lemma would fall into the scope of Proposition \ref{prop_bounded_length_of_hyperbolic_geodesic} if it was nonpositively curved. Hence these particular examples must have some positive curvature.
\begin{proof}
Let $\eps>0$. Let $\gamma$ be a systole for $\sigma$ on $M$, i.e., such that $l_{\sigma}(\gamma)=\sys(\sigma)$. We assume that $\eps<\sys(\sigma)$. For any $\delta$, to be specified later, we consider a smooth bump function $\phi\colon \mathbb R \rightarrow \mathbb R$ such that $\phi(0)=1$ and $\phi(t) = 0$ for every $t\notin (-\delta, \delta)$.
We define a smooth function $u\colon M\to \R$ by
\[
u(x) = -\log\left(\frac{\sys(\sigma)}{\eps}\right)\phi(d_{\sigma}(x,\gamma))+C(1-\phi(d_{\sigma}(x,\gamma)))
\]
where $C$ is some real number that will be determined later and $d_{\sigma}(x, \gamma)$ is the $\sigma$-distance from a point $x$ to the curve $\gamma$. In particular, $u$ is equal to $-\log\left(\frac{\sys(\sigma)}{\eps}\right)$ on $\gamma$ and equal to $C$ outside the $\delta$-neighborhood (for $\sigma$) of $\gamma$.

Set $g:=e^{2u}\sigma$. For an appropriate choice of $\delta$ and $C$, we will show that the metric $g$ satisfy the lemma. 

First, notice that the $g$-length of $\gamma$ is equal to $\eps$. This implies that $\sys(g)\leqslant \eps$. As $g$ is in the conformal class of $\sigma$, all we have left to show is that the $g$-volume of $M$ is $A$, and that the $g$-diameter of $M$ is at most twice the $\sigma$-diameter.

Call $V(\gamma,\delta)$ the $\delta$-neighborhood of $\gamma$ for the hyperbolic metric $\sigma$. For sufficiently small $\delta$ we have $\vol_{\sigma}\left(V(\gamma,\delta)\right)\leq A/2$. Then, the $g$-volume of $M$, which is equal to $\int\limits_{M}e^{2u}dv_{\sigma}$, satisfy

\begin{align*}
 \int_{M}e^{2u}dv_{\sigma} &\geqslant e^{2C}\vol_{\sigma}\left(M \smallsetminus V(\gamma,\delta) \right)\geqslant \frac{e^{2C}A}{2}>A \quad \text{if } C> \frac{\log 2}{2} \\
 \int_{M}e^{2u}dv_{\sigma} &\leqslant e^{2C}\vol_{\sigma}\left(M \smallsetminus V(\gamma,\delta)\right) + \vol_{\sigma}\left(V(\gamma,\delta)\right)\leqslant e^{2C}A+\frac{A}{2}<A \quad \text{if } C<-\frac{\log 2}{2}.
\end{align*}

Therefore, using the above inequalities and their continuous dependency on $C$, we can deduce that there exists $C(\delta)$ such that $\vol_{g}(M) = A$. Moreover, $C(\delta)$ converges to zero when $\delta$ does.

Given the construction, it is also clear that, as $\delta$ goes to zero, the diameter of the metric $g$ tends to the diameter of $\sigma$. Hence, we can choose $\delta$ small enough so that our last claim is satisfied.
\end{proof}

If we do not require the condition that $g$ has to be conformally equivalent to $\sigma$ in Lemma~\ref{Example_systole}, then by Section 2.1 in \cite{EK} we have the following
\begin{lem}\label{lem_small_systole_small_diameter}
Suppose $M$ is a surface of genus $\geqslant 2$ and $A>0$. Then, there exists $D>0$ such that, for every $\eps>0$, there exists a negatively curved Riemannian metric $g$ on $M$ of total area $A$, diameter less than $D$, and the $g$-length of its systole is smaller than $\eps$.
\end{lem}

Here, the hyperbolic metrics in the conformal class of the examples of Lemma \ref{lem_small_systole_small_diameter} must escape every compact of the Teichm\"uller space. Indeed, they are build in such a way (see \cite[section 2.1]{EK}) that Proposition \ref{prop_bounded_length_of_hyperbolic_geodesic} can apply to each of them, i.e, their systole is along a geodesic of the hyperbolic metric in their conformal class. Hence, if the hyperbolic metrics in their conformal class were to stay in a compact of the Teichm\"uller space, then Proposition \ref{prop_bounded_length_of_hyperbolic_geodesic} would give a lower bound on the systole, a contradiction with Lemma \ref{lem_small_systole_small_diameter}.

\subsection{Examples of metrics with arbitrary large diameter}
\begin{lem}\label{Example_diameter}
Let $M$ be a surface of genus $\geqslant 2$, equipped with a hyperbolic metric $\sigma$, of total area $A$. Then, for every $D>0$, there exists a Riemannian metric $g$ on $M$ of total area $A$ that is conformally equivalent to $\sigma$ and $\diam(M,g)>D$.
\end{lem}
\begin{proof} The gist of the construction is quite easy: Consider a small neighborhood around some point and change the metric conformally in such a way that this neighborhood is stretched, but without changing the volume too much. As a result one gets a surface with a long and very thin nose.
The only difficulty is to actually produces a function that has exactly the right type of behavior.

As before, we fix a point $p$ on $M$ and consider standard hyperbolic polar coordinates $(r, \theta)$. Let us consider a family of smooth bump functions $\phi(r; a)\colon \mathbb R \rightarrow \mathbb R$, where $a>0$, such that $\phi(r; a)=1$ when $r\in[-a/2; a/2]$ and $\phi(r; a)=0$ when $r\geqslant a$. Fix a positive constant $\eps>0$ such that the $\eps$-neighborhood (for $\sigma$) of $p$ is inside the coordinate chart. Then, for any positive $\delta<-\frac{1}{\log\eps-\log 2}$ (i.e., $\delta$ is such that $\left(\frac{\eps}{2}\right)^{-\delta}<e$) we define a smooth function $\rho(x;\delta, C)$ on $M$ by, for every $x\in M$ inside the coordinate chart 
\begin{multline} \label{equation_definition_of_rho}
\rho(x;\delta, C) = \phi(r; \eps)\left(\frac{\eps}{2}\right)^{-\delta/2}\left(\frac{A\delta}{16\pi}\right)^{\frac{1}{2}}\left(\phi(r;e^{-\frac{1}{\delta}})\left(\frac{e^{-\frac{1}{\delta}}}{2}\right)^{-\left(1-\frac{\delta}{2}\right)}+(1-\phi(r;e^{-\frac{1}{\delta}}))r^{-\left(1-\frac{\delta}{2}\right)}\right) \\
+(1-\phi(r; \eps))C,
\end{multline}
and $\rho(x;\delta, C)=C$ for every $x\in M$ outside the coordinate chart.
As in the proof of Lemma \ref{Example_systole}, $C$ is a positive constant that will be chosen so that the total area of $M$ for the metric $g_{\delta} = \rho^2(x;\delta, C)\sigma$ is equal to $A$.

Let $B(p,\eps)$ denote the $\eps$-neighborhood of $p$ for the hyperbolic metric $\sigma$. Notice that for sufficiently small $\eps$ we have $\vol_{\sigma}\left(B(p,\eps)\right)\leq A/2$. Then, the $g$-volume of $M$, which is equal to $\int\limits_{M}\rho^2dv_{\sigma}$, satisfy

\begin{align*}
 \int_{M}\rho^2(x;\delta, C)dv_{\sigma} &\geqslant C^2\vol_{\sigma}\left(M \smallsetminus B(p,\eps) \right)\geqslant \frac{C^2A}{2}>A \quad \text{if } C^2> 2 \\
 \int_{M}\rho^2(x;\delta, C)dv_{\sigma} &\leqslant C^2\vol_{\sigma}\left(M \smallsetminus B(p,\eps)\right) + \vol_{g_{\delta}}\left(B(p,\delta)\right)\leqslant C^2 A+\frac{A}{2}<A \quad \text{if } C^2<\frac{1}{2}.
\end{align*}

Therefore, using the above inequalities and their continuous dependency on $C$, we deduce that there exists $C(\delta,\eps)$ such that $\vol_{g_{\delta}}(M) = A$. Moreover, $C(\delta,\eps)$ converges to $1$ when $\eps$ tends to zero.

The inequality $\vol_{g_{\delta}}\left(B(p,\delta)\right)\leqslant \frac{A}{2}$, for $\eps$ small enough, is obtained by direct computations. Indeed, we have
\begin{multline}\label{area_ball}
Area_{g_{\delta}}(B(p,\eps/2)) = \int\limits_0^{2\pi}\int\limits_0^{\eps/2}\rho^2(r;\delta)\sinh(r)\,dr\,d\theta\leqslant \int\limits_0^{2\pi}\int\limits_0^{\eps/2}\left(\frac{\eps}{2}\right)^{-\delta}\left(\frac{A\delta}{8\pi}\right)r^{-\left(1-\delta\right)}\,dr\,d\theta = \\ =
\left(\frac{\eps}{2}\right)^{-\delta}\left(\frac{A\delta}{4}\right)\frac{1}{\delta}\left(\frac{\eps}{2}\right)^{\delta}\leqslant \frac{A}{4}, 
\end{multline}
for sufficiently small $\eps$. % We used the fact that $\sinh r<2r$ for sufficiently small $r$.

And, if $C^2<1$ and $\eps, \delta$ are sufficiently small, the area of the annulus $\{(r, \theta)| r\in[\eps/2; \eps], \theta\in [0;2\pi)\}$ satisfies
\begin{multline*}
\int\limits_0^{2\pi}\int\limits_{\eps/2}^{\eps} \rho^2(r,\delta)\sinh(r)drd\theta\leqslant 4\pi\int\limits_{\eps/2}^{\eps} \left(\left(\frac{\eps}{2}\right)^{-\frac{\delta}{2}}\left(\frac{A\delta}{16\pi}\right)^{\frac{1}{2}}r^{-(1-\frac{\delta}{2})}+C\right)^2rdr \\
\leqslant 4\pi\left(\frac{A(2^{\delta}-1)}{16\pi}+C^2\frac{\eps^2}{2}+\frac{2C\eps}{2^{-\frac{\delta}{2}}(1+\frac{\delta}{2})}\left(\frac{A\delta}{16\pi}\right)^{\frac{1}{2}}\right)\leqslant \frac{A}{4}.
\end{multline*}

So all we have left to do is to prove that the diameter of $g_{\delta}$ is greater than $D$.
First, notice that the radial curves (for $\sigma$) in the annulus $B_{\sigma}(\frac{\eps}{2})\setminus B_{\sigma}(e^{\frac{1}{\delta}})$ are minimal geodesics. Indeed, since the conformal coefficient in $B_{\sigma}(\frac{\eps}{2})\setminus B_{\sigma}(e^{\frac{1}{\delta}})$ depends only on $r$ and not on $\theta$, the shortest curves between the boundaries are those such that the angle component stays constant, i.e., radial curves.

Now, let $l$ be such a radial curve in $B_{\sigma}(\frac{\eps}{2})\setminus B_{\sigma}(e^{\frac{1}{\delta}})$. Then its $g_{\delta}$-length satisfies
\begin{align}
length_{g_{\delta}}(l) = \int\limits_{e^{-\frac{1}{\delta}}}^{\eps/2}\rho(r;\delta)dr & \geqslant \int\limits_{e^{-\frac{1}{\delta}}}^{\eps/2}r^{-\left(1-\frac{\delta}{2}\right)}\left(\frac{\eps}{2}\right)^{-\delta/2}\left(\frac{A\delta}{16\pi}\right)^{\frac{1}{2}}dr \nonumber \\
&\geqslant \left(\frac{\eps}{2}\right)^{-\delta/2}\left(\frac{A\delta}{16\pi}\right)^{\frac{1}{2}}\frac{2}{\delta}\left(\left(\frac{\eps}{2}\right)^{\delta/2}-e^{-\frac{1}{2}}\right) \nonumber \\
&\geqslant \frac{1}{\delta^{\frac{1}{2}}}\left(\frac{A}{4\pi}\right)^{\frac{1}{2}}\left(1-\left(\frac{\eps}{2}\right)^{-\delta/2}e^{-\frac{1}{2}}\right) \rightarrow +\infty \qquad \text{ as } \quad \delta\rightarrow 0.  \label{annulus_new_size}
\end{align}

%It follows from the fact that if $\gamma(t) = (r(t), \theta(t))$, where $t\in[c,d]$ and $\theta(c)=\theta(d)=\theta_0$, $r(c)=a$ and $r(d)=b$, then

%\begin{multline*}
%length_{g_{\delta}}(\gamma) = \int\limits_c^d\sqrt{\rho^2(r(t);\delta)\left(\frac{dr(t)}{dt}\right)^2+\rho^2(r(t);\delta)\sinh^2(r(t))\left(\frac{d\theta}{dt}\right)^2}dt\geqslant\\\geqslant \int\limits_c^d\sqrt{\rho^2(r(t);\delta)\left(\frac{dr(t)}{dt}\right)^2}dt = \int\limits_a^b\rho(r;\delta)dr = length_{g_\delta}(\tilde\gamma). 
%\end{multline*} 

Therefore, for every positive constant $D$ there exists a positive number $\delta$ such that the diameter of $M$ with respect to a Riemannian metric $g_{\delta}$ (defined above) of total area $A$ and which is conformally equivalent to $\sigma$ is larger than $D$.
\end{proof}

\begin{rem} \label{rem_ColboisElSoufi}
 Notice that one can prove the above result more theoretically but less explicitely. The idea, used in \cite{CE03}, goes as follows: A Riemannian surface $(\Sigma,g)$ is, locally, almost Euclidean. Hence, it is, locally, almost conformal to the sphere with its standard metric. Therefore, starting with any metric $h$ conformal to the standard sphere, it is possible to construct a conformal deformation of $(\Sigma,g)$ around a point to make it arbitrarily close to $h$ in that neighborhood. 
 Applying that scheme when $h$ is a long and thin sphere gives an example as in Lemma \ref{Example_diameter}. One can also use this method to build conformal Cheeger Dumbell. Again, we will instead give an explicit example in the proof of Lemma \ref{Example_eigenvalue}.
\end{rem}

We recall that if we do not require the condition that $g$ has to be conformally equivalent to $\sigma$ in Lemma~\ref{Example_diameter} then there exists a Riemannian metric of constant negative curvature of fixed total area such that its diameter is larger than any a priori given positive number.

\subsection{Examples of metrics with arbitrary small first eigenvalue of Laplacian}

Our next example shows that if one drops the nonpositive curvature assumption, then Corollary \ref{cor_metric_entropy_vs_lambda_1} fails, i.e., one can get arbitrarily small $\lambda_1$ in a fixed conformal class.
\begin{lem}\label{Example_eigenvalue}
Let $M$ be a surface of genus $\geqslant 2$, equipped with a hyperbolic metric $\sigma$, of total area $A$. Then, for every $\eps>0$ there exists a Riemannian metric $g$ of total area $A$ that is conformally equivalent to $\sigma$ and $\lambda_1(g)<\eps$.
\end{lem}

\begin{rem}
 As mentioned in Remark \ref{rem_ColboisElSoufi} above, such examples can also be obtained using the techniques of Colbois and El Soufi in \cite{CE03}.
\end{rem}

\begin{proof}
Let $p, q$ be two points in $M$ at $\sigma$-distance strictly greater than $\eps$. We define a metric $g_{\delta}$ that has total area $A$ and is conformally equivalent to $\sigma$ in the following way. 

Let $\rho(\cdot ;\delta)$ be the function that is constant outside of $\eps$-neighborhoods of $p$ and $q$ and defined by equation \eqref{equation_definition_of_rho} inside of the neighborhoods. We define $g_{\delta} := \rho^2(\cdot;\delta)\sigma$.

%\begin{align*}
%\rho(x;\delta) = \left(\frac{\eps}{2}\right)^{-\delta/2}\left(\frac{A\delta}{16\pi}\right)^{\frac{1}{2}}\left(\phi(r;e^{-\frac{1}{\delta}})\left(\frac{e^{-\frac{1}{\delta}}}{2}\right)^{-\left(1-\frac{\delta}{2}\right)}+(1-\phi(r;e^{-\frac{1}{\delta}}))r^{-\left(1-\frac{\delta}{2}\right)}\right)
%\end{align*} 

%for all $x$ in the $\eps/2$-neighborhood of $p$ or $q$, where $r$ is the $\sigma$-distance from $x$ to $p$ or $q$. Outside of $\eps$-neighborhoods of $p$ and $q$ we have $\rho(x;\delta)$ equal to some positive constant. {\Red How is $\rho$ defined in the annulus $\eps/2$ to $\eps$?}

 An argument similar to the proof of Lemma~\ref{Example_diameter} shows that one can pick the value of the function $\rho$ outside of the $\eps$-neighborhoods in such a way that $g_{\delta}$ has total area $A$. 

Recall from Lemma~\ref{Example_diameter} that a $\sigma$-ball centered at $p$ or $q$ of radius smaller than or equal to $\eps/2$ is a $g_\delta$-ball a priori of some other radius. Let $(R, t)$ be polar coordinates for $g_{\delta}$ in the neighborhood of $p$. Denote by $R_1$ and $R_2$ the $g_{\delta}$-radii of $\sigma$-balls of radii $e^{-\frac{1}{\delta}}$ and $\eps/2$, respectively. Notice that equation \eqref{annulus_new_size} gives
\begin{align*}
R_2-R_1 = \frac{1}{\delta^{\frac{1}{2}}}\left(\frac{A}{4\pi}\right)^{\frac{1}{2}}\left(1-\left(\frac{\eps}{2}\right)^{-\delta/2}e^{-\frac{1}{2}}\right) \rightarrow +\infty \qquad \text{ as } \quad \delta\rightarrow 0
\end{align*}

Let $f_1$ be a function on $M$ such that it is equal to $1$ in the $R_1$-neighborhood of $p$ (for $g_\delta$) and $0$ outside of the $R_2$-neighborhood of it. In the annulus $\{(R,t)| R\in[R_1; R_2], t\in[0;2\pi)\}$ we define $f_1(R,t) = -\frac{R}{R_2-R_1}+\frac{R_2}{R_2-R_1}$. Let $f_2$ be the function obtained as $f_1$, but considering $q$ instead of $p$. Notice that $f_1, f_2\in H^1(M)$ as they are continuous and piecewise continuously differentiable. Furthermore, $f_1$ and $f_2$ have disjoint supports by construction.

The fact that the first eigenvalue of Laplacian for $g_{\delta}$ tends to $0$ as $\delta$ tends to $0$ will follow from the Min-Max principle (see equation \eqref{min-max}).
Let $V_2$ be the $2$-dimensional subspace of the Sobolev space $H^1(M)$ spanned by the functions $f_1$ and $f_2$. Any $F\in V_2$ can be written as $F = af_1+bf_2$, where $a,b\in\mathbb R$. Then,
\begin{align*}
\int\limits_M F^2 dv_{g_{\delta}} = a^2\int\limits_M f^2_1dv_{g_{\delta}}+2ab\int\limits_M f_1f_2dv_{g_{\delta}}+b^2\int\limits_M f^2_2dv_{g_{\delta}} = a^2\int\limits_M f^2_1dv_{g_{\delta}}+b^2\int\limits_M f^2_2dv_{g_{\delta}},
\end{align*}
as $f_1,f_2$ have disjoint supports. Similarly,
\begin{align*}
\int\limits_M |\nabla_{g_{\delta}} F|_{g_{\delta}}^2 dv_{g_{\delta}} = a^2\int\limits_M |\nabla_{g_{\delta}} f_1|_{g_{\delta}}^2dv_{g_{\delta}}+b^2\int\limits_M |\nabla_{g_{\delta}} f_2|_{g_{\delta}}^2dv_{g_{\delta}}.
\end{align*}
Therefore, 
\begin{align*}
R_{g_{\delta}}(F) = \frac{\int\limits_{M}|\nabla_{g_{\delta}} F|_{g_{\delta}}^2\,dv_{g_{\delta}}}{\int\limits_{M}f^2\,dv_{g_{\delta}}} = \frac{a^2\int\limits_M |\nabla_{g_{\delta}} f_1|_{g_{\delta}}^2dv_{g_{\delta}}+b^2\int\limits_M |\nabla_{g_{\delta}} f_2|_{g_{\delta}}^2dv_{g_{\delta}}}{a^2\int\limits_M f^2_1dv_{g_{\delta}}+b^2\int\limits_M f^2_2dv_{g_{\delta}}}\leqslant R_{g_{\delta}}(f_1)+R_{g_{\delta}}(f_2).
\end{align*}

The Min-Max principle and the estimate above on $R_{g_{\delta}}(F)$ imply that $\lambda_1(g_{\delta})\leqslant R_{g_{\delta}}(f_1)+R_{g_{\delta}}(f_2)$. As a result, if we show that $R_{g_{\delta}}(f_1)$ and $R_{g_{\delta}}(f_2)$ tend to $0$ as $\delta$ tends to $0$, then $\lambda_1(g_{\delta})$ tends to $0$ as $\delta$ tends to $0$ and Lemma~\ref{Example_eigenvalue} is proven.

We only estimate $R_{g_{\delta}}(f_1)$, since the computations for $f_2$ are the same.
% $\int\limits_M |\nabla_{g_{\delta}} f_1|_{g_{\delta}}^2dv_{g_{\delta}}$ and $\int\limits_M f_1^2dv_{g_{\delta}}$,  

Let $S(R_1;R_2) = \{(R,t)|r\in[R_1;R_2], t\in[0;2\pi)\}$. By definition of $f_1$ and equation \eqref{area_ball}, for sufficiently small $\eps$ we have
\begin{align*}
\int\limits_M |\nabla_{g_{\delta}} f_1|_{g_{\delta}}^2dv_{g_{\delta}} = \frac{1}{(R_2-R_1)^2}Area_{g_{\delta}}(S(R_1;R_2)) \leqslant \frac{A}{4(R_2-R_1)^2}.
\end{align*}

Notice, since $R_2-R_1$ tends to $+\infty$ as $\delta$ tends to $0$, that $\int\limits_M |\nabla_{g_{\delta}} f_1|_{g_{\delta}}^2dv_{g_{\delta}}$ tends to $0$ as $\delta$ tends to $0$.

So all we have left to do is show that $\int_M f_1^2 dv_{g_{\delta}}$ is bounded below by a strictly positive constant.

For $R_2>R>R_1$, the coordinates $R$ and $r$, which are respectively the $g_{\delta}$- and $\sigma$-distances from $p$, are related by the following equation.

%{\Red From there I would be inclined to write: direct computations shows that $f_1(R,t) = bla$ and that $\int_M f_1^2 dv_{g_{\delta}} >0$.}
%First, we find the connection between $R$ and $r$, which are $g_{\delta}$ and $\sigma$ distances from the point $p$ to some other point. Let us consider points such that $R_2>R>R_1$. Then, 
\begin{align*}
R = R_1 + \int\limits_{e^{-\frac{1}{\delta}}}^r \left(\frac{\eps}{2}\right)^{-\delta/2}\left(\frac{A\delta}{16\pi}\right)^{\frac{1}{2}}r^{-\left(1-\frac{\delta}{2}\right)}dr = R_1 + \left(\frac{\eps}{2}\right)^{-\delta/2}\left(\frac{A}{4\pi\delta}\right)^{\frac{1}{2}}\left(r^{\frac{\delta}{2}}-e^{-\frac{1}{2}}\right).
\end{align*}

So direct computation shows that, for $R_2>R>R_1$,
\begin{equation*}
f_1(R,t) = \frac{R_2-R}{R_2-R_1} = \frac{\left(\frac{\eps}{2}\right)^{\frac{\delta}{2}} - r^{\frac{\delta}{2}}}{\left(\frac{\eps}{2}\right)^{\frac{\delta}{2}}-e^{-\frac{1}{2}}} 
\end{equation*}
and there exists a positive constant $F=F(A)$ such that $\int\limits_M f_1^2\,dv_{g_\delta}>F(A)$ for sufficiently small $\delta$.
Therefore, $R_g(f_1)$ tends to $0$ as $\delta$ tends to $0$, and the same is true for $f_2$. 
\end{proof}

It is easy to construct hyperbolic metrics with arbitrarily small $\lambda_1$, just by leaving compacts in the Teichm\"uller space. However, these examples have also unbounded diameter. Here, we show that one can build negatively curved surface with bounded diameter and arbitrarily small $\lambda_1$. Notice that by Corollary \ref{cor_metric_entropy_vs_lambda_1}, the hyperbolic metrics in the conformal class of these examples must also leave every compacts of the Teichm\"uller space.
\begin{lem}\label{lem_lambda_not_conform}
Suppose $M$ is a surface of genus $2$ and $A>0$. There exists $D>0$ such that, for every $\eps>0$, there exists a negatively curved Riemannian metric $g$ of total area $A$, the diameter less than $D$, and $\lambda_1(g)<\eps$.
\end{lem}

\begin{proof}
Let $\sigma$ be a metric on $M$ of constant negative curvature and total area $A$. Let $\gamma$ be a closed $\sigma$-geodesic such that $l_{\sigma}(\gamma) =\sys(\sigma)$. We assume that we chose $\sigma$ such that it is symmetric with respect to $\gamma$ and has a rotationally invariant cylindrical neighborhood $C$ of sufficiently large size around $\gamma$. In \cite{EK}, it was shown that it is possible to modify $\sigma$ on this cylindrical neighborhood and obtain a family of negatively curved metrics of total area $A$ such that the infimum of the systole in this family of metrics is $0$ and the supremum of the diameter in it is bounded above by a constant $D$, which depends only on $A$ and on the initial choice of $\sigma$. We recall that construction and point out what modifications are needed in order to get Lemma~\ref{lem_lambda_not_conform}.

 Without loss of generality, we may assume that $\sigma$ has curvature $-1$ (so $A= 2\pi|\chi(M)|$). Let $a:=\sys(\sigma)/2\pi$. The metric tensor in the normal coordinates $(r, \theta)$ for $\sigma$ with respect to the $\sigma$-systole has the form $\left(\begin{matrix} 1 & 0\\ 0 & a^2\cosh^2(r)\end{matrix}\right)$, where $\theta$ is the coordinate along the $\sigma$-systole, which runs from $0$ to $2\pi$, and $r$ is the distance to the $\sigma$-systole.

 The metric $g$ that we build matches with $\sigma$ outside of the cylindrical neighborhood $C$ of $\gamma$ and is also symmetric and rotationally invariant in it. In general, if $g$ is a smooth metric that admits a rotationally invariant cylindrical neighborhood of a closed geodesic, then in the normal coordinates $(r,\theta)$ with respect to that geodesic the metric tensor for $g$ has the form $\left(\begin{matrix} 1 & 0\\ 0 & f^2(r)\end{matrix}\right)$. Moreover, $g$ is negatively curved if and only if $f''(r)>0$, that is, $f(r)$ is convex.
 
 It is easy to construct (see Figure \ref{fig_small_lambda1}) a function $f(r)$ that is convex, even, matches smoothly $a\cosh(r)$ outside of a fixed neighborhood of $0$, and such that $f(r)$ is arbitrarily small on another neighborhood of $0$ of fixed length. The metric $g$ obtained from $f$ is then negatively curved, coincides with $\sigma$ outside of $C$, and contains an almost flat cylinder of fixed area and arbitrarily small radius. We can further normalize $g$ so that the total volume is $A$, without affecting any of the essential features of $g$: $\diam(M,g)$ is uniformly bounded by some constant $D$, negatively curved and contains an almost flat cylinder of fixed area and arbitrarily small radius (see \cite[Section 2.1]{EK} for more details).
 
 Hence the classical proof that the Cheeger dumbbell has arbitrarily small $\lambda_1$ applies to $g$ and yields the lemma (see, for instance \cite{BGM}).\qedhere

%  
% In \cite[Section 2.1]{EK}, it is not principle that $f(r)$ is small only along the $\sigma$-geodesic, to define $f(r)$ with $f''(r)>0$ and which coincide with $a\cosh(r)$ outside of piece of fixed length which depends only on the curvature of $\sigma$. In particular, we can make $f(r)$ with $f''(r)>0$ which is arbitrarily small on a cylinder of fixed length (see Figure \ref{fig_small_lambda1}).  Moreover, we can take this cylinder to be symmetric with respect to the closed geodesic. In particular, the area of the cylinder can be made arbitrarily small. After normalization of the total area, we obtain a family of negatively curved metrics metrics with total area $A$ and symmetric with respect to a systole such that the infimum of the area of the fixed cylinder in this family of metrics is $0$ and the supremum of the diameter in it is bounded above by a constant $D$, which depends only on $A$ and the initial choice of $\sigma$. See Section 2.1 of \cite{EK} for more details. Basically, we have the Cheeder Dumbell example and the argument for it (see \cite{C}) for our case show that the infimum of the first eigenvalue of the Laplacian in the constructed family of metric is $0$.

\begin{figure}[H]
      \centering
      \includegraphics[height=4.5cm]{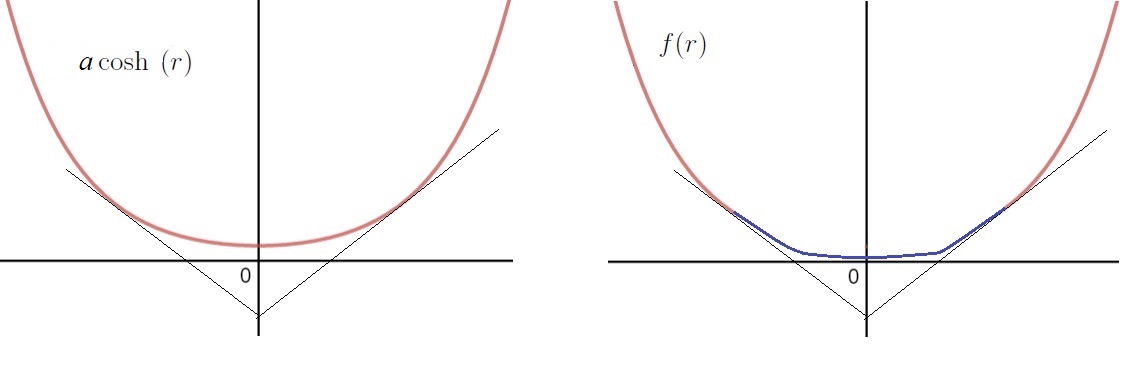}
      \caption{The example of $f(r)$.}
      \label{fig_small_lambda1}
      \end{figure}
\end{proof}

\section{Topological entropy vs systole}

As we mentioned in the introduction, putting together Theorem \ref{bound_entr_systole} below and a result of Besson, Courtois and Gallot (\cite[Corollaire 0.6]{BCG_margulis}) implies that, for a family of negatively curved metrics of fixed total area with diameter uniformly bounded above, then the topological entropy goes to infinity if and only if the systole goes to zero.

We were not able to find a reference for Theorem \ref{bound_entr_systole} in the literature, therefore, we provide a proof here, based on suggestions by Keith Burns.
%Since Burns' result is not available in the literature, with his kind authorization, we provide the statement and his proof here
\begin{thm}\label{bound_entr_systole}
Suppose $M$ is a smooth closed Riemannian manifold of dimension $n$ with negative sectional curvature and total volume $V$. There exists an upper bound on the topological entropy of the geodesic flow that depends only on the injectivity radius of the metric, $n$ and $V$. 
\end{thm}

\begin{proof}
Let $\rho$ be the injectivity radius of $M$ and set $\eps = \rho/4$. Then, the length of the shortest closed geodesic is $2\rho = 8\eps$. The volume of a ball of radius $\eps$ in $M$ is bounded from below by the volume of a ball of radius $\eps$ in ${\mathbb R}^n$, i.e, $\nu(\eps,n) = (\eps\sqrt{\pi})^n/\Gamma(\frac{n}{2} + 1)$. Let $N = \lfloor{V/\nu(\eps,n)}\rfloor$.

There can be at most $N$ pairwise disjoint $\eps$ balls in $M$. Since the centers of a maximal collection of disjoint $\eps$ balls in $M$ are a $2\eps$-spanning set, it is possible to choose $N$ points, $p_1,\dots,p_N$ such that the balls $B(p_i,2\eps)$ cover $M$. For each $p \in M$ choose $j(p) \in \{1,\dots,N\}$ such that $p \in B(p_{j(p)},2\eps)$. 

A geodesic  $\gamma\colon [0,T] \to M$ can be coded by the sequence 
$$
j(\gamma(0)), j(\gamma(\eps)), j(\gamma(2\eps)),\dots, j(\gamma(n_\eps(T)\eps), j(\gamma(T)),
$$
where $n_\eps(T)$ is the largest integer such that $n_\eps(T)\eps < T$. Suppose that $\gamma_1\colon [0,T_1] \to M$ and 
 $\gamma_2\colon [0,T_2] \to M$ are two closed geodesics with the same coding sequences. For each $n$ the geodesic segments $\gamma_1|_{[n\eps,(n+1)\eps]}$ and $\gamma_2|_{[n\eps,(n+1)\eps]}$ lie in a common ball of radius $3\eps < \rho$. The same is true of the segments $\gamma_1|_{[n_\eps(T_1)\eps, T_1]}$ and $\gamma_2|_{[n_\eps(T_2)\eps, T_2]}$. It follows that the Hausdorff distance between the two geodesics is at most $\rho$. Consequently, the two closed geodesics must be the same (up to reparametrization).
 
Hence, the number of closed geodesics of length $\leqslant T$ will be at most the number of possible codings, which is bounded above by $N^{(T/\eps + 1)}$. Now,
\[
 \lim_{T\to +\infty}\frac{1}{T} \log \left(N^{(T/\eps + 1)}\right) = \frac{1}{\eps}\log N.
\]
Since $g$ is negatively curved, the topological entropy is given by the exponential growth rate of the number of closed geodesics of length $\leqslant T$ \cite{B72}, \cite{M69}. Hence $(\log N)/\eps$ is an upper bound on the topological entropy and depends only on $V$, $n$ and $\eps =\rho/4$.
\end{proof}

\end{document}